\documentclass[11pt]{article}
\usepackage[oldstylenums]{kpfonts}
\usepackage{array}
\usepackage{geometry}
\usepackage{graphicx}
\usepackage{amsmath,amsfonts,amssymb,amsthm,mathrsfs,dsfont}
\usepackage{hyperref}
\usepackage[active]{srcltx}
\usepackage[T1]{fontenc}
\usepackage{color}
\usepackage{esint}
\usepackage{enumitem}
\usepackage{appendix}

\numberwithin{equation}{section}

\newtheorem{theo}{Theorem}[section]

\newtheorem{lem}[theo]{Lemma}
\newtheorem{cor}[theo]{Corollary}
\newtheorem{prop}[theo]{Proposition}
\theoremstyle{remark}
\newtheorem{rmk}[theo]{Remark}

\theoremstyle{definition}
\newtheorem{defi}[theo]{Definition}

\allowdisplaybreaks

\newcommand{\eps}{\varepsilon}
\newcommand{\rd}{\mathrm{d}}
\newcommand{\qqtext}[1]{\qquad\mbox{#1}\qquad} % big space text[1] big space
\newcommand{\R}{\mathbb{R}}
\newcommand{\Z}{\mathbb{Z}}
\newcommand{\N}{\mathbb{N}}

\newcommand{\dist}{\operatorname{dist}}

\newcommand{\mrest}{  \,\raisebox{-.127ex}{\reflectbox{\rotatebox[origin=br]{-90}{$\lnot$}}}\,} %restriction of a measure
\renewcommand{\SS}{\mathsf{S}} % for the gradient-flow semi-group
\renewcommand{\d}{{\mathrm d}}			%derivative	%derivative
\newcommand{\dt}{{\d t}}
\newcommand{\ddt}{{\frac \d\dt}}

\newcommand{\sfd}{{\sf d}}				%distance
\newcommand{\X}{{\rm X}}				%state space
\newcommand{\E}{{\sf E}}
\newcommand{\sfn}{{\sf n}}				%outer normal vector
\renewcommand{\Z}{{\cal Z}}
\newcommand{\EVI}{{\rm EVI}}
\newcommand{\Dir}{{\rm Dir}}

\newcommand{\RCD}{{\rm RCD}}

\newcommand{\loc}{{\rm loc}}
\newcommand{\Dom}{\operatorname{D}}
\DeclareMathOperator*{\esssup}{ess\,sup}

\begin{document}

\title{Convex functions defined on metric spaces are pulled back to subharmonic ones by harmonic maps}
\date{}
\author{Hugo Lavenant
	\thanks{Department of Decision Sciences and BIDSA, Bocconi University, Via Roberto Sarfatti, 25, 20100 Milano MI Italy.
		email: hugo.lavenant@unibocconi.it
	}
	\and L\'eonard Monsaingeon
	\thanks{GFM Universidade de Lisboa, Campo Grande, Edif\'icio C6, 1749-016 Lisboa, Portugal
		and IECL Universit\'e de Lorraine, F-54506 Vandoeuvre-l\`es-Nancy Cedex, France.
		email: leonard.monsaingeon@univ-lorraine.fr
	}
	\and Luca Tamanini
	\thanks{CEREMADE (UMR CNRS 7534), Universit\'e Paris Dauphine PSL, Place du Mar\'echal de Lattre de Tassigny, 75775 Paris Cedex 16, France
		and INRIA-Paris, MOKAPLAN, 2 Rue Simone Iff, 75012, Paris, France.
		email: tamanini@ceremade.dauphine.fr}
	\and Dmitry Vorotnikov
	\thanks{University of Coimbra, CMUC, Department of Mathematics, 3001-501 Coimbra, Portugal.
		email: mitvorot@mat.uc.pt
	}
}
\maketitle
\begin{abstract}
If $u : \Omega\subset \R^d \to \X$ is a harmonic map valued in a metric space $\X$ and $\E : \X \to \R$ is a convex function, in the sense that it generates an $\EVI_0$-gradient flow, we prove that the pullback $\E \circ u : \Omega \to \R$ is subharmonic.
This property was known in the smooth Riemannian manifold setting or with curvature restrictions on $\X$, while we prove it here in full generality.
In addition, we establish generalized maximum principles, in the sense that the $L^q$ norm of $\E \circ u$ on $\partial \Omega$ controls the $L^p$ norm of $\E \circ u$ in $\Omega$ for some well-chosen exponents $p \geq q$, including the case $p=q=+\infty$.
In particular, our results apply when $\E$ is a geodesically convex entropy over the  Wasserstein space, and thus settle some conjectures of \emph{Y. Brenier. Extended Monge-Kantorovich theory. In Optimal transportation and applications (Martina Franca, 2001), volume 1813 of Lecture Notes in Math., pages 91-121. Springer, Berlin, 2003}.
\end{abstract}
%\tableofcontents
%%%%%%%%%%%%%%%%%%%%%%%%%%%%%%%%%%%%%%%%%%%%%%%%%%%%%%%%%%%%%%%%%%%%%%%%%%%%%%%%%%%%%%%%%%%%%%%%%%%%%%%%%%%%%%%%%%%%%%%%%%%%%%%%%%%%%%%%%%%%%%%%%%%%%%%%%%%%%%%%
\vspace{1cm}


\textbf{MSC [2020] Primary 58E20, 58E35 Secondary 31C45, 53C23}

\section{Introduction}

%We start with the following observation. 

Let $\Omega\subset \R^d$, $D\subset \R^n$ be open domains, $u:\Omega\to D$ be a smooth \emph{harmonic map}, i.e.\ $\Delta u = 0$, and $\E:D\to \R$ be a smooth \emph{convex function}, i.e.\ $\nabla^2 \E \geq 0$.  A simple computation shows that $u$ pulls $\E$ back to a \emph{subharmonic function} on $\Omega$, i.e.\ $-\Delta (\E \circ u) \leq 0$. Consequently, $u(\partial \Omega)\subset C$ implies $u(\Omega)\subset C$ for any closed convex subset $C\subset D$, which can be viewed as a maximum principle, cf.\ \cite{Fu05} (just letting $\E$ be the distance to $C$). 

Ishihara \cite{Ish79} proved that the property of pulling (geodesically) convex functions defined on a Riemannian manifold (without boundary) back to subharmonic functions defined on another such manifold is actually \emph{equivalent} to the harmonicity of the corresponding map between the manifolds, see also \cite[Chapter 9]{Jost17}. Several authors generalized this result to more abstract settings.
Namely, the above-mentioned equivalence is known provided either (i) the target $\X$ of the alleged harmonic map  $u:\Omega\to \X$ is  a metric tree or a Riemannian manifold of non-positive curvature (NPC in short) \cite{Sturm05}, or (ii) when the source domain $\Omega$ is a $1$-dimensional Riemannian polyhedron and the target $\X$ is a locally compact geodesic NPC metric space \cite{EF01}, and finally (iii) when the the domain $\Omega$ is a Riemannian polyhedron of arbitrary dimension but the target $\X$ is either a smooth Riemannian manifold without boundary or a Riemannian polyhedron of non-positive curvature of dimension $\leq 2$, cf.\ \cite{EF01, Fu05}.

In this work we will only concentrate on the ``direct'' statement\footnote{For the reasons already mentioned the property of pulling convex functions back to subharmonic ones is sometimes referred to as the \emph{Ishihara property}.}, that is:
\begin{equation}
\label{statement_ishihara}
[ u:\Omega \to \X \text{ harmonic and } \E : \X \to \R \text{ convex} ] \Rightarrow [ \E \circ u \text{ subharmonic on } \Omega ]
\end{equation}
This was established by Fuglede \cite[Theorem 2]{Fu05} provided the domain $\Omega$ is a Riemannian polyhedron, the target $\X$ is a simply connected complete geodesic NPC metric space, and $\E:\X\to \R$ is a continuous convex function.

Recent advances tackled a particular metric space $\X$ that does not admit any upper bound on the curvature, namely, when $\Omega$ is a bounded open domain $\Omega\subset \R^d$  and $\X=\mathcal P_2(D)$ is the space of probability measures over a convex compact set $D\subset \R^n$ equipped with the $2$-Wasserstein distance (the latter is known to be a positively curved metric space -- PC, in short -- see \cite{AmbrosioGigliSavare08}).
In this framework the statement \eqref{statement_ishihara} was conjectured by Y.~Brenier in \cite{B01} when $\E : \mathcal{P}_2(D) \to [0,+\infty]$ is the Boltzmann entropy $H(\rho)=\int_D\rho\log\rho\, \rd x$, which is geodesically convex on $\mathcal P_2(D)$ \cite{Villani03, AmbrosioGigliSavare08}.
The first author gave recently a partial positive answer to this conjecture in \cite{HL19}:
he actually proved that, given a geodesically convex function $\E:\mathcal P_2(D)\to [0,+\infty]$ and a boundary condition $u^b:\partial \Omega\to \mathcal P_2(D)$, one can find at least \emph{one} solution $u:\overline\Omega\to \mathcal P_2(D)$ of the Dirichlet problem with boundary conditions $u^b$ such that $-\Delta(\E \circ u)\leq 0$ in $\Omega$ in the sense of distributions.
This does not automatically imply that the Ishihara property \eqref{statement_ishihara} holds for \emph{every} harmonic map because the uniqueness of the solutions to the corresponding Dirichlet problem has not yet been established.
He also provided a generalized maximum principle: namely, the essential supremum of the harmonic pullback to $\overline \Omega$ of a convex entropy on the Wasserstein space is always attained at $\partial \Omega$.
Notably, at this level of generality this does not follow immediately from the subharmonicity of the pullback since the latter one is subharmonic only in $\Omega$ and is not necessarily continuous up to the boundary.

\textbf{In this paper we are concerned with the generalization of the claim \eqref{statement_ishihara} when $\Omega$ is a smooth bounded domain of $\R^d$ but $\X$ is an abstract metric space, without any (local) compactness or curvature-related restrictions}.
We also establish several generalized maximum principles for the harmonic pullback.
These read roughly speaking as
\begin{equation}
\label{eq:max_principle_informal}
\| \E(u) \|_{L^p(\Omega)} \leq C \| \E(u) \|_{L^q(\partial \Omega)}
\end{equation}
for some $p\geq q$, with $C$ depending only on $p,q$, and $\Omega$.
In such control, the finiteness of the left-hand side is a consequence of that of the right-hand side. The classical maximum principle corresponds to $p = q = + \infty$ and $C = 1$.

At this level of generality, there exist various ways of defining harmonicity and even convexity. For the latter we impose the stronger condition that $\E$ is not only convex, but moreover generates an $\EVI_0$-gradient flow in the sense of \cite{MS20}.
This encompasses the cases of geodesically convex functions on NPC spaces, geodesically convex functions on $\mathrm{RCD}(K,\infty)$ spaces, and also for a large class of entropy functionals over the Wasserstein space $\X = \mathcal{P}_2(Y)$ where $Y$ is a complete geodesic metric space, see \cite[Section 3.4]{MS20} for precise definitions and statements in all theses cases.
Moreover, according to \cite[Theorem 3.15]{MS20}, in a forthcoming work it will be shown that geodesically convex functions on complete PC metric spaces always generate an $\EVI_0$-gradient flow.
In short, the gap between our assumption (that $\E$ generates an $\EVI_0$-gradient flow) and the geodesic convexity of $\E$ is tiny and subtle.
On the other hand, for the notion of harmonicity we will adopt a purely \emph{variational} approach and look only at maps $u : \Omega \to \X$ which are \emph{global} minimizers of a Dirichlet energy in the spirit of Korevaar and Schoen \cite{KS93}. Note that without additional assumptions on the target space $\X$, such as negative curvature, harmonic maps with prescribed boundary conditions are not expected to be unique.

The main strategy will be to use the gradient flow generated by $\E$ in order to suitably perturb the map $u$, and then exploit quantitative information from the defect of optimality of this perturbation.
The most technical part comes from modifying the function $u$ up to the boundary of $\Omega$ in order to prove \eqref{eq:max_principle_informal}. 
The technique of proof (using an $\EVI_0$-gradient flow to perturb a metric-valued map) was first used by the second author and Baradat in \cite{baradat2020small} when $\Omega$ is one-dimensional and $\E$ is the Boltzmann entropy on the Wasserstein space, based on probabilistic arguments for fluid-mechanical applications.
Again for one-dimensional domains $\Omega=[0,1]$, a similar approach was employed by some of the authors in the setting when $\X$ is the non-commutative Fisher-Rao space or, more generally, an abstract metric space, cf.\ \cite{MV20,monsaingeon2020dynamical}.
Our strategy here can be extended to \emph{local} minimizers of the Dirichlet energy, see Remark~\ref{rk:local_min}, but would fail when leaving the variational framework (that is, if considering arbitrary critical points of the Dirichlet energy).

\begin{rmk}[Brenier's conjectures] 
	Brenier \cite{B01} first conjectured that the pullback of the Boltzmann entropy acting on the Wasserstein space by a harmonic map is subharmonic, and gave a formal proof of this claim.
	As a vague corollary, he surmised the following second conjecture, cf.\ \cite[Conjecture 3.1]{B01}. Take $u : \Omega \to \mathcal P_2(D)$ a harmonic map valued in the Wasserstein space (see \cite{HL19} which justifies that the notion of harmonic map in \cite{B01} coincides with the one of the present work). Brenier was expecting that if the measures $u(x)$ for $x \in \partial \Omega$ are absolutely continuous with respect to the Lebesgue measure, then it is also the case for $u(y)$ with $y \in \Omega$. 
	Taking $\E$ to be any displacement convex internal energy \cite{Villani03}, e.g. the Boltzmann or the R\'enyi--Tsallis entropy, we give here a positive answer (for a.e. $y\in \Omega$) to Brenier's second conjecture under the assumption that $\E(u) \in L^1(\partial \Omega)$, see our Theorem~\ref{thm:main}. 
	In particular, Theorem~\ref{thm:main} also settles the first conjecture in full generality, improving the partial result of \cite[Theorem 6.3]{HL19}.
	Let us point out that Brenier defined probability-measure-valued harmonic functions by generalizing the Monge-Kantorovich optimal transport, but it has been shown in \cite{HL19} that his definition is compatible with the abstract framework that we are adopting here. 
\end{rmk}

The paper is organized as follows.
In Section \ref{s2}, we rigorously specify our framework and state the main results.
Our most fundamental conclusion is Theorem \ref{thm:main}.
It not only implies the subharmonicity of the pullback inside $\Omega$, but also deals with the boundary values and leads to the corresponding $L^1$, $L^p$ and $L^\infty$ maximum principles, see respectively Theorems~\ref{thm:main}, ~\ref{rk:Lpmax}, ~\ref{th:Linfty}.
In Section \ref{s3}, we discuss the properties of $\EVI_0$-gradient flows, harmonic maps and Dirichlet energies to be used throughout.
The proofs are postponed to Section \ref{s4}.
Appendix finally \ref{ap1} contains two technical real-analysis lemmas.

\section{Main statements}
\label{s2}

\subsection*{Setting}
Let us start by fixing all the assumptions on the main objects of our study, mainly the domain $\Omega$, the target $\X$, and the functional $\E$.

If $a \in \R$ we denote by $a^+ = \max\{a,0\}$ and $a^- = \max\{-a,0\}$ its positive and negative part, respectively.

We always assume for simplicity that $\Omega\subset \R^d$ is an open bounded domain of class at least $C^{2,\alpha}$, see Remark~\ref{rk:domain_regularity} for a comment about domains with less regular boundaries.
We endow $\Omega$ with its Lebesgue measure and its boundary $\partial \Omega$ with the $\mathcal{H}^{d-1}$ measure. 

As concerns the target space and the functional, let $(\X,\sfd)$ be a complete metric space and let $\E : \X \to \R \cup \{+\infty\}$ be lower semicontinuous with dense domain, i.e.\ $\overline{\operatorname D(\E)} = \X$, where $\operatorname D(\E) = \{ u \in \X \,:\, \E(x) < + \infty \}$. 
We do \emph{not} assume that $\E$ is bounded from below.
Moreover, we assume that $\E$ generates an $\EVI_0$-gradient flow for any initial condition $u \in \X$, in the sense of \cite{AmbrosioGigliSavare08}.
This means that there exists an everywhere-defined $1$-parameter semigroup $(\SS_t)_{t \geq 0}$ such that $t \mapsto \SS_t u$ belongs to $AC_\loc((0,\infty),\X) \cap C([0,\infty), \X)$ and
\begin{equation}
\label{eq:EVI0}
\tag{$\EVI$}
\frac{1}{2}\ddt \sfd^2(\SS_t u,v) + \E(\SS_t u) \leq \E(v), \qquad\forall v \in \X,\,\textrm{a.e. } t>0.
\end{equation}
In particular, this implies that $\E$ is geodesically convex \cite{DaneriSavare08,monsaingeon2020dynamical}.

\medskip

Let us also provide the crucial definition for the understanding of the statements of the main results: the notion of harmonic map.
This requires introducing first the space of metric-valued Sobolev functions, the Dirichlet energy, and the notion of trace.

The space $L^2(\Omega,\X)$ consists of (equivalence classes up to a.e.\ equality of) Borel maps $u : \Omega \to \X$ with separable essential range, i.e.\ $u(\Omega \setminus N) \subset \X$ is separable for some $N \subset \Omega$ with $|N|=0$, such that $\sfd(u(\cdot),v) \in L^2(\Omega,\R)$ for some (hence for any, since $|\Omega| < \infty$) $v \in \X$. For $u \in L^2(\Omega,\X)$ the $\eps$-approximate Dirichlet energy is defined as
\begin{equation}\label{eq:dirichlet_eps}
\Dir_\eps(u) := \frac{1}{2 C_d \eps^{d+2}}\iint_{\Omega \times \Omega} \sfd^2(u(x),u(y))\mathds{1}_{|x-y| \leq \eps}\,\d x\d y
,\qquad C_d := \frac{1}{d} \int_{B_1} |z|^2 \,\rd z, 
\end{equation}
and the Dirichlet energy as its limit as $\eps \downarrow 0$ (which always exists, being possibly $+\infty$, see \cite[Theorem 4]{Chiron07}), namely
\begin{equation}
\label{eq:dirichlet}
\Dir(u) := \lim_{\eps \to 0}\Dir_\eps(u).
\end{equation}
The space $H^1(\Omega,\X)$ is then defined as $\{u \in L^2(\Omega,\X) \,:\, \Dir(u) < \infty\}$ and, as shown in \cite[Theorem 1.12.2]{KS93}, there exists a well-defined trace from $H^1(\Omega,\X)$ into $L^2(\partial\Omega,\X)$ that we shall denote by
\begin{equation}\label{eq:trace-op}
\begin{array}{ccc}
H^1(\Omega,\X) & \to & L^2(\partial\Omega,\X)\\
u & \mapsto & u^b
\end{array}.
\end{equation}
This is built as follows: given a vector field $\Z$ transversal to $\partial\Omega$, let $(x_t)_t$ be the flow induced by $\Z$ starting at $x_0 \in \partial\Omega$, i.e.\ the unique solution to $\dot{x}_t = \Z(x_t)$ with $x|_{t=0} = x_0$; then $u$ admits a representative such that, for a.e.\ $x_0\in\partial\Omega$, the map $t \mapsto u(x_t)$ is H\"older continuous and thus $u^b(x_0) := \lim_{t \downarrow 0} u(x_t)$ exists. It turns out that this construction depends neither on the choice of the representative nor on the transversal field $\Z$. Given any $\varphi \in H^1(\Omega,\X)$, we set $H^1_\varphi(\Omega,\X) := \{u \in H^1(\Omega,\X) \,:\, u^b = \varphi^b\}$.

\begin{rmk}
In the case $(\X,\sfd)$ is a smooth Riemannian manifold, the Dirichlet energy is nothing but the $L^2$-norm of the differential of $u$
\[
\Dir(u) = \frac{1}{2} \int_\Omega \|\d u(x)\|_{\operatorname F}^2\,\d x,
\]
being $\| \cdot \|_{\operatorname F}$ the Frobenius (a.k.a.\ Hilbert-Schmidt) norm.
Moreover, for a smooth $\E$, the existence of an $\EVI_0$-gradient flow of $\E$ simply means that $\E$ is convex, and in this case the curve $t \mapsto \SS_t u = \gamma_t$ is the solution to the gradient flow equation $\dot{\gamma}_t = -\nabla\E(\gamma_t)$ with initial condition $\gamma|_{t=0} = u$.
\end{rmk}

\begin{defi} \label{d:har} A map $u \in H^1(\Omega,\X)$ %with boundary trace $u^b$ 
is said to be \emph{harmonic} if it is a minimizer of the Dirichlet problem
\[
\inf_{v \in H^1_u(\Omega,\X)} \Dir(v).
\] 
\end{defi}

\begin{rmk} 
There exist many ways to define harmonic maps. In particular, those maps might be merely required to be locally minimizing, or just be critical points of the Dirichlet energy, etc., leading to different qualitative behavior already when $(\X,\sfd)$ is a smooth Riemannian manifold~\cite{helein2008harmonic}.
In this paper we consider only harmonicity in the sense of Definition \ref{d:har} (see however Remark~\ref{rk:local_min} for an easy weakening), which can be seen as the most restrictive notion.
For the ease of exposition we shall simply talk of harmonic maps.
\end{rmk}

\subsection*{The main results}

Our first main result is the Ishihara property for harmonic maps as defined above. 

\begin{theo}[Ishihara property]
\label{thm:main}
Let $u \in H^1(\Omega,\X)$ be harmonic with boundary value $u^b$ such that $\E(u^b)\in L^1(\partial \Omega)$.
Then $\E(u)\in L^1(\Omega)$ and
\begin{equation}
\label{eq:control_energy_entropy_improved_harmonic}
\int_\Omega(-\Delta \varphi)(x)\E(u(x))\,\d x 
\leq 
\int_{\partial\Omega}\left(-\frac{\partial\varphi}{\partial\sfn}\right)\E( u^b)\,\d\sigma
\quad\forall \varphi\in C^2(\overline{\Omega}),
\,\varphi\geq 0,\,\varphi|_{\partial\Omega}=0,
\end{equation}
where $\sigma := \mathcal{H}^{d-1}\mrest \partial\Omega$ and $\sfn$ is the outward unit normal to $\partial\Omega$. In particular, $\E(u)$ is subharmonic in the distributional sense.
\end{theo}

The identity \eqref{eq:control_energy_entropy_improved_harmonic} can be interpreted as
\[
\begin{cases}
 -\Delta \E(u) \leq 0 & \mbox{in }\Omega,\\
 \E(u)\leq \E(u^b) & \mbox{on }\partial\Omega
\end{cases}
\]
in the weak sense, and it enables us to prove some generalized maximum principles. 

Primarily, we derive a ``classical'' $L^\infty$-maximum principle estimate. 

\begin{theo}[$L^\infty$-maximum principle]
\label{th:Linfty}
Let $u \in H^1(\Omega,\X)$ be harmonic with boundary value $u^b$ such that $\E(u^b)\in L^1(\partial\Omega)$.
Then
\[
\esssup\limits_{x\in\Omega}\E(u(x))
\leq \esssup\limits_{x\in\partial\Omega} \E(u^b(x)),
\]
and the left-hand side is finite if and only if the right-hand side is.
\end{theo} 

Ultimately, leveraging the classical theory of elliptic regularity, we get the refinement of Theorem \ref{thm:main} that was informally anticipated in \eqref{eq:max_principle_informal}. 

\begin{theo}[Gain of $L^p$-regularity]
\label{rk:Lpmax}
Let $u \in H^1(\Omega,\X)$ be harmonic with boundary value $u^b$ such that $\E(u^b)^+\in L^{q}(\partial \Omega)$, $1<q<\infty$.
Then $\E(u)^+\in L^{p}(\Omega)$ with the explicit exponent $p=\frac {dq} {d-1}$. Moreover, if $q=1$ (exactly as in Theorem ~\ref{thm:main}), then $\E(u)^+\in L^{p}(\Omega)$ for any $p<\frac d {d-1}$. 
\end{theo}

\begin{rmk}
\label{rmk:E_sublinear} In Theorem~\ref{thm:main}, the conclusion $\E(u) \in L^1(\Omega)$ will be derived from the integrability of $\E(u^b)$, which can be viewed as an $L^1$ maximum principle, cf.\ \eqref{eq:max_principle_informal}. Actually, as detailed later in Lemma~\ref{lem:E_L1}, although $\E$ can be unbounded from below, there always holds $\E(u)^- \in L^2(\Omega)$ and $\E(u^b)^- \in L^2(\partial \Omega)$, this is a direct consequence of $u \in L^2(\Omega,\X)$ and $u^b \in L^2(\partial\Omega,\X)$.
In Theorem \ref{rk:Lpmax} we independently control the positive part of $\E(u)$. In particular, if $d>1$ and $\E(u^b)^+\in L^{1}(\partial \Omega)$, the above considerations yield $\E(u)\in L^{p}(\Omega)$ for any $p<\frac d {d-1}$. 

\end{rmk}

\begin{rmk}
	Even in the case when $\X$ is an NPC space, our analysis slightly improves the existing results.
	In particular, in comparison with \cite{Fu05} we do not assume that $\E$ is continuous, which is crucial for infinite-dimensional applications of the abstract metric theory.  
\end{rmk}

\section{Some useful properties of EVI-gradient flows and Dirichlet energies}
\label{s3}

For the sake of the reader, and as a complement to the essential definitions already provided in the previous section, we collect here all properties concerning $\EVI$-gradient flows, harmonic maps and Dirichlet energies that are required in the sequel.

\subsection*{Properties of EVI-gradient flows} We list now some useful properties of $\EVI$-gradient flows, which hold true under the aforementioned assumptions on $\X$ and $\E$. First of all, the slope of $\E$ (defined as a local object) admits the following global representation
\begin{equation}
\label{eq:local=global}
|\partial\E|(u) := \limsup_{v \to u}\frac{\big(\E(u) - \E(v)\big)^+}{\sfd(u,v)} = \sup_{v \neq u}\frac{\big(\E(u) - \E(v)\big)^+}{\sfd(u,v)},
\end{equation}
provided $u \in \Dom(\E)$, see \cite[Proposition 3.6]{MS20} taking into account our standing assumption that any $u \in \X$ is the starting point of an $\EVI_0$-gradient flow.
This implies in particular that $|\partial\E|:\X \to [0,\infty]$ is lower semicontinuous, since so is the above right-hand side (as a supremum of lower semicontinuous functions).
Moreover, as an easy byproduct of \eqref{eq:local=global} we see that
\begin{equation}\label{eq:doubled-slope}
|\E(u)-\E(v)| \leq \big(|\partial \E|(u)+|\partial \E|(v)\big)\sfd(u,v), \qquad \forall u,v \in \Dom(\E).
\end{equation}
In addition, from \cite[Theorems 2.17 and 3.5]{MS20} we know that:
\begin{enumerate}[label={(\roman*)}]
\item \emph{Contraction}.\ If $(\gamma_t)$ is an $\EVI_0$-gradient flow of $\E$ starting from $u \in \overline{\operatorname D(\E)}$ and $(\tilde{\gamma}_t)$ is a second $\EVI_0$-gradient flow of $\E$ starting from $v \in \overline{\operatorname D(\E)}$, then
\begin{equation}
\label{eq:contraction}
\sfd(\gamma_t,\tilde{\gamma}_t) \leq \sfd(u,v), \qquad \forall t \geq 0.
\end{equation}
This means that EVI-gradient flows are unique (provided they exist) and thus if there exists an EVI-gradient flow $(\gamma_t)$ starting from $u$, then a 1-parameter semigroup $(\SS_t)_{t \geq 0}$ is unambiguously associated to it via $\SS_t(u) = \gamma_t$.

%\item The maps $t \mapsto \gamma_t$ and $t \mapsto \V(\gamma_t)$ are locally Lipschitz in $(0,\infty)$ with values in $\X$ and $\R$, respectively, and satisfy the \emph{Energy Dissipation Equality}
%\begin{equation}\label{eq:speed=slope}
%-\ddt \V(\gamma_t) = \frac{1}{2}|\dot{\gamma}_t|^2 + \frac{1}{2}|\partial \V|^2(\gamma_t) = |\dot{\gamma}_t|^2 = |\partial \V|^2(\gamma_t), \qquad \textrm{for a.e. } t>0.
%\end{equation}

\item \emph{Monotonicity}.\ For any $u \in \X$, the map
\begin{align}
\label{eq:monotonicity-entropy}
& t \mapsto \E(\SS_t u) \quad \textrm{is non-increasing on } [0,\infty); \\
\label{eq:monotonicity-slope}
& t \mapsto |\partial \E|(\SS_t u) \quad \textrm{is non-increasing on } [0,\infty).
\end{align}
%\item For any $u \in \X$, $v \in \operatorname D(\E)$ and $t>0$ it holds
%\begin{equation}\label{eq:regularization}
%\frac{1}{2}\sfd^2(\SS_t u,v) + t\big(\E(\SS_t u) - \E(v)\big) + \frac{t^2}{2}|\partial\E|^2(\SS_t u) \leq \frac{1}{2}\sfd^2(u,v).
%\end{equation}
\item \emph{Regularizing effect for the energy}.\ For any $u \in \X$, $v \in \operatorname D(\E)$ and $t>0$ it holds
\begin{equation}\label{eq:regularization}
\E(\SS_t u)  \leq \E(v) + \frac{1}{2t}\sfd^2(u,v).
\end{equation}
\item \emph{Regularizing effect for the slope}.\ For any $u \in \X$, $v \in \operatorname D(|\partial\E|)$ and $t>0$ it holds
\begin{equation}\label{eq:regularization-slope}
|\partial \E|^2(\SS_t u) \leq |\partial \E|^2(v) + \frac{1}{t^2}\sfd^2(u,v).
\end{equation}
\item \emph{Control of the speed}.
\ For any $u \in \operatorname D(|\partial\E|)$ and $t > 0$ there holds $\sfd(\SS_t u,u) \leq t |\partial\E|(u)$.
We will rather use this in the following weaker form:
if $u \in \X$ and $t_1,t_2 > 0$ then
\begin{equation}\label{eq:asymptotics}
\sfd(\SS_{t_1} u, \SS_{t_2} u) \leq |t_1-t_2| \big( |\partial\E|(\SS_{t_1}u) + |\partial\E|(\SS_{t_2}u)\big).
\end{equation}
\item \emph{Bound from below}.\ $\E$ is linearly bounded from below, namely there exist $v \in \X$, $\alpha,\beta \in \R$ such that
\begin{equation}\label{eq:linear-lower-bound}
\E(u)  \geq \alpha -  \beta\sfd(u,v), \qquad \forall u \in \X.
\end{equation}
\end{enumerate}
A further fundamental estimate for $\EVI_0$-gradient flows is the following.
It can be seen as a refinement of the contractivity property \eqref{eq:contraction}, allowing to compare the distance between two $\EVI_0$-gradient flows at \emph{different} times.

\begin{lem}
\label{lem:fundamental}
For all $u_1,u_2\in \X$ and all $t_1,t_2 > 0$ let $v_1 := \SS_{t_1}u_1$, $v_2 := \SS_{t_2}u_2$. Then
\begin{equation}\label{eq:2_points_estimate}
\frac 12\sfd^2(v_1,v_2) + (t_1-t_2)\left(\E(v_1)-\E(v_2)\right) \leq \frac 12 \sfd^2(u_1,u_2).
\end{equation}
\end{lem}

\noindent This was proved in \cite{monsaingeon2020dynamical} for $\EVI_\lambda$-gradient flows with arbitrary $\lambda\in \R$ in a slightly different form, for the sake of completeness we include here the full proof in the simpler case $\lambda=0$.
%As a preliminary remark, note that in \eqref{eq:2_points_estimate} $\E(v_1)$ and $\E(v_2)$ may take value $+\infty$. However, this can only occur if $s_1=0$ or $s_2=0$, respectively, since e.g.\ by \eqref{eq:regularization-slope} we see that $\SS_t$ maps $\X$ into $\operatorname D(|\partial\E|) \subset \operatorname D(\E)$ whenever $t>0$. If $\E(v_1)=\E(v_2)=+\infty$, then $s_1=s_2=0$ whence $v_i=\SS_{s_i}u_i=\SS_0 u_i=u_i$ and thanks to our convention $0\times\infty=0$ the left-hand side of \eqref{eq:2_points_estimate} boils down to $\frac 12\sfd^2(u_1,u_2)$, thus coinciding with the right-hand side. If $\E(v_1)<+\infty$ and $\E(v_2)=+\infty$, then $s_2=0$ and without loss of generality we can assume $s_1>0$ (since we have already agreed on the meaning of \eqref{eq:2_points_estimate} when $s_1=s_2=0$); in this case the statement is trivially true, as the left-hand side is $-\infty$ owing to $(s_1-s_2)>0$ and $\E(v_1)-\E(v_2)=-\infty$. The case $\E(v_2)<+\infty$ and $\E(v_1)=+\infty$ is completely analogous. 
 
\begin{proof}
As a preliminary remark, note e.g.\ from \eqref{eq:regularization} that $\SS_t$ maps $\X$ into $\operatorname D(\E)$ whenever $t>0$, so that $\E(v_1)$ and $\E(v_2)$ are both finite and thus \eqref{eq:2_points_estimate} is unambiguous. Moreover, if $t_1=t_2=t$ then \eqref{eq:2_points_estimate} is nothing but the contractivity property \eqref{eq:contraction} thus we only need to establish the claim when $t_1 \neq t_2$.

Consider $t_2>t_1$ (the other case is completely symmetric) and write \eqref{eq:EVI0} for the gradient flow $t\mapsto \SS_t u_2$ with reference point $v=v_1=\SS_{t_1}u_1$ in the form
\[
\frac 12 \frac{\d}{\d t}\sfd^2(\SS_t u_2,v_1) + \E(\SS_t u_2)-\E(v_1)\leq 0.
\]
Integrating from $t=t_1$ to $t=t_2>t_1$ we get
\begin{align*}
\frac 12 \sfd^2(\SS_{t_2} u_2,v_1)  
& +\int _{t_1}^{t_2}
\big(\E(\SS_t u_2) -\E(v_1)\big)\d t
\\
& \leq
\frac 12 \sfd^2(\SS_{t_1} u_2,v_1)  
= \frac 12 \sfd^2(\SS_{t_1} u_2,\SS_{t_1}u_1)  \leq \frac 12 \sfd^2(u_2,u_1),
\end{align*}
where the last inequality follows from \eqref{eq:contraction}. Leveraging now the monotonicity property \eqref{eq:monotonicity-entropy} we see that $\E(\SS_t u_2) \geq \E(\SS_{t_2}u_2)=\E(v_2)$ for $t\leq t_2$.
Whence
\[
\int_{t_1}^{t_2} \big(\E(\SS_t u_2) -\E(v_1)\big)\d t 
\geq 
\int_{t_1}^{t_2} \big(\E(v_2) -\E(v_1)\big)\d t
=
(t_2-t_1)\big(\E(v_2) -\E(v_1)\big)
\]
and plugging this estimate in the previous inequality yields exactly \eqref{eq:2_points_estimate}.
\end{proof}

\subsection*{Korevaar-Schoen theory}
Since the seminal works of N.\ J.\ Korevaar and R.\ M.\ Schoen \cite{KS93}, and J.\ Jost \cite{Jost94} on Sobolev and harmonic maps from Riemannian manifolds into metric spaces, several other (equivalent) approaches have appeared, most notably those of Y.\ G.\ Reshetnyak \cite{Reshetnyak97} and P.\ Haj{\l}asz \cite{Hajlasz09}, and in the last years there has been a surge of interest concerning the generalization of these papers to singular/non-smooth source spaces, including in particular Alexandrov and $\RCD$ spaces.
In this direction it is worth mentioning \cite{KuwaeShioya03, ZhangZhu18, GigTyu20, GigTyu20bis}. 

Our definition of $\eps$-approximate Dirichlet energy is, up to a dimensional factor, a particular case of that introduced in \cite[Section 1.5]{Chiron07}, inspired by \cite{BBM01}, for the explicit choice of (non-renormalized) radial mollifiers $\rho_\eps(x) := \eps^{-(d+2)}|x|^2\mathds{1}_{|x| \leq \eps}$. 
By \cite[Theorem 4]{Chiron07} the limit of $\Dir_\eps(u)$ as $\eps \downarrow 0$ always exists in $[0,\infty]$ for all $u \in L^2(\Omega,\X)$, so that \eqref{eq:dirichlet} is actually meaningful as already mentioned.
Furthermore, by \cite[Theorems 3 and 4]{Chiron07} $\Dir$ coincides with the Dirichlet energies defined in \cite{KS93, Jost94}; thus by \cite[Theorem 1.6.1]{KS93} we know that
\begin{equation}\label{eq:lsc-dirichlet}
\Dir \textrm{ is lower semicontinuous w.r.t.\ the strong topology of } L^2(\Omega,\X),
\end{equation}
where it is worth recalling that this topology is induced by the distance
\begin{equation}\label{eq:L2-distance}
\sfd_{L^2}(u,v) := \Big(\int_\Omega \sfd^2(u(x),v(x))\,\d x\Big)^{1/2}, \qquad u,v \in L^2(\Omega,\X).
\end{equation}
As a consequence of \cite[Theorems 3 and 4]{Chiron07}, the Sobolev space $H^1(\Omega,\X)$ as defined in Section \ref{s2} also coincides with those introduced in \cite{KS93, Jost94}, and by \cite[Proposition 4]{Chiron07} together with \cite[Proposition 5.1.6]{ambrosio2004topics} it also coincides with Reshetnyak's and Haj{\l}asz's definitions.
The equivalence with the latter implies in particular that $u \in H^1(\Omega,\X)$ if and only if there exists a negligible set $N \subset \Omega$ and a non-negative function $g_u \in L^2(\Omega)$ such that
\begin{equation}\label{eq:hajlasz}
\sfd(u(x),u(y)) \leq \big(g_u(x) + g_u(y)\big)|x-y|, \qquad \forall x,y \in \Omega \setminus N.
\end{equation}

\bigskip

\section{Proof of the main results}
\label{s4}

This section is devoted to the proof of Theorems \ref{thm:main}--\ref{rk:Lpmax}. 
%As already mentioned, Theorem~\ref{th:Linfty} and Theorem~\ref{rk:Lpmax} will be deduced from Theorem ~\ref{thm:main} by the mean of dual formulation of $L^p$ norms and classical analysis. 
The non-smooth analysis work is done in the proof of Theorem \ref{thm:main}.
The main idea is to perturb $u$ by looking at the function $\tilde{u}^\lambda = \SS_{\lambda \phi(x)} u(x)$ for some $\phi : \Omega \to [0, + \infty)$ (to be well-chosen):
in other words, we let $u(x)$ follow the gradient flow of $\E$ for a ``time'' $\lambda \phi(x)$ that depends on $x \in \Omega$.
Applying Lemma~\ref{lem:fundamental} at the level of the $\varepsilon$-approximate Dirichlet energies, and then taking the limit $\varepsilon \to 0$, a formal computation reveals that we should be able to estimate $\Dir(\tilde{u}^\lambda)$ as
\begin{equation}
\label{eq:control_dir_informal}
\Dir(\tilde{u}^\lambda) + \lambda \int_\Omega \nabla \phi \cdot \nabla \E(\tilde{u}^\lambda)\,\d x \leq \Dir(u).
\end{equation} 
If $\phi$ vanishes on $\partial \Omega$ then $\tilde{u}^\lambda$ shares the same boundary values as $u$, thus by the minimizing property of $u$, we deduce that $\int_\Omega \nabla \phi \cdot \nabla \E(\tilde{u}^\lambda)\,\d x \leq 0$.
An integration by parts and the limit $\lambda \to 0$ should then readily imply our main result.
However this computation is difficult to justify at once because $\E\circ \tilde{u}^\lambda:\Omega\to\R$ is \emph{a priori} not that smooth.
This is why we add an additional regularization parameter $\delta > 0$ and rather look, say for $\lambda = 1$ fixed, at $\SS_{\delta + \phi(x)} u(x)$:
we ``lift'' everything up in time uniformly by $\delta$.
Owing to the regularizing effects \eqref{eq:regularization} and \eqref{eq:regularization-slope} we prove that the functions $\SS_{\delta + \phi} u$ and $\E(\SS_{\delta + \phi} u)$ are smooth enough to justify our subsequent computations (Lemma~\ref{lem:E_L1} and Lemma~\ref{lem:E_W11}), then we establish the control \eqref{eq:control_dir_informal} of the Dirichlet energy of $\SS_{\delta + \phi} u$ (Proposition~\ref{prop:control_energy_entropy}) and we then take the limit $\delta \to 0$ (Proposition~\ref{prop:phi_superharmonic}) followed by $\lambda \to 0$ (yielding Theorem~\ref{thm:main}).
Eventually, the integrability of $\E(u)$ is obtained by duality: having first established in Proposition~\ref{prop:phi_superharmonic} the estimate \eqref{eq:control_energy_entropy_improved_harmonic} when the test function $\phi$ is superharmonic and relying on monotone convergence, a well-chosen test function yields integrability of $\E$, see Corollary~\ref{cor:E_is_L1}, as well as Theorem~\ref{th:Linfty} and Theorem~\ref{rk:Lpmax} thanks to classical elliptic regularity for the (standard) Poisson equation.
\\

We now start the core of the proof.

\begin{lem}
\label{lem:E_L1}
Let $(\Xi, \mathcal A, \mu)$ be a measure space with $\mu(\Xi)<+\infty$, and $u \in L^2(\Xi,\X)$. Take $\delta > 0$ and $\phi : \Xi \to \R_+$ a non-negative and bounded function. Set 
\[
\tilde u^\delta(x):=\SS_{\delta+\phi(x)}u(x), \qquad \tilde u(x):=\SS_{\phi(x)}u(x).
\]
Then $\tilde u^\delta$ and $\tilde u$ belong to $L^2(\Xi,\X)$. In addition, $\E(u)^-$, $\E(\tilde u)^-$ and $\E(\tilde u^\delta)^-$ belong to $L^2(\Xi)$ while $\E(\tilde{u}^\delta)^+$ belongs to $L^1(\Xi)$. 
\end{lem}

\begin{proof}
Let us first justify the integrability of $\tilde u^\delta$ and $\tilde u$. By triangle inequality and by the contraction estimate \eqref{eq:contraction} with $v = \SS_{\delta + \phi(x)}w$ for some $w \in \X$ fixed,
\begin{align}\label{eq:uniform-bound}
\sfd(\SS_{\delta+\phi(x)}u(x),w) & \leq \sfd(\SS_{\delta+\phi(x)}u(x),\SS_{\delta+\phi(x)}w) + \sfd(\SS_{\delta+\phi(x)}w,w) \nonumber \\ 
& \leq \sfd(u(x),w) +  \sfd(\SS_{\delta+\phi(x)}w,w).
\end{align}
Due to the continuity-in time of the gradient flow, the second term is bounded uniformly in $x \in \Xi$ by a constant that depends only on $\| \phi \|_\infty + \delta$ and $w$, whence $\tilde{u}^\delta \in L^2(\Xi,\X)$. Moreover, nothing prevents here from taking $\delta = 0$, which yields $\tilde u \in L^2(\Xi,\X)$ too.  

In order to justify that $\E(u)^-$, $\E(\tilde u)^-$ and $\E(\tilde u^\delta)^-$ belong to $L^2(\Xi)$ we simply use the lower bound \eqref{eq:linear-lower-bound} and that $u, \tilde u^\delta$ and $\tilde u$ belong to $L^2(\Xi,\X)$.

Eventually, for the integrability of $\E(\tilde{u}^\delta)$ for $\delta > 0$ we observe that \eqref{eq:regularization} implies that,  for any point $v \in \operatorname D(\E)$, we have
\begin{equation*}
\E(\SS_{\delta+\phi(x)} u(x)) \leq \E(v) + \frac{1}{2 (\delta + \phi(x))} \sfd^2(\SS_{\delta + \phi(x)} u(x), v) \leq \E(v) + \frac{1}{2 \delta} \sfd^2(\SS_{\delta + \phi(x)} u(x), v),
\end{equation*} 
and the right-hand side is in $L^1(\Xi)$ since we just established that $\tilde{u}^\delta \in L^2(\Xi,\X)$.
\end{proof}

\begin{lem}
\label{lem:E_W11}
Let $u \in H^1(\Omega,\X)$ with boundary value $u^b$ be given. Fix $\delta>0$ and $\phi \in C^1(\overline{\Omega})$ non-negative with $\phi = 0$ on $\partial\Omega$. Set 
\[
\tilde u^\delta(x):=\SS_{\delta+\phi(x)}u(x),
\qqtext{and}
\tilde E^{\delta}(x):=\E(\tilde u^\delta(x)).
\]
Then $\tilde E^{\delta}\in W^{1,1}(\Omega)$ with boundary trace
\[
(\tilde E^\delta)^b(x)=\E(\SS_\delta u^b(x)),
\hspace{1cm}x\in \partial\Omega.
\]
\end{lem}

\begin{proof}
Our claim that $\tilde E^\delta \in L^1(\Omega)$ is a direct consequence of Lemma~\ref{lem:E_L1}.
We turn next to to the control of the derivative. 
As $u \in H^1(\Omega,\X)$, by \eqref{eq:hajlasz} there exists $g_u \in L^2(\Omega)$ and a negligible set $N \subset \Omega$ such that $\sfd(u(x),u(y)) \leq (g_u(x) + g_u(y)) |x-y| $ for all $x,y \notin N$.
On the other hand, fix an arbitrary $w \in \X$.
Then, due to \eqref{eq:regularization-slope} with $v = \SS_{\delta + \phi(x)} w$, for all $x \in \Omega$ there holds
\begin{align}\label{eq:control_dE_delta}
|\partial \E|^2(\SS_{\delta + \phi(x)} u(x)) & \leq |\partial \E|^2(\SS_{\delta + \phi(x)} w) + \frac{1}{(\delta + \phi(x))^2}\sfd^2(\SS_{\delta + \phi(x)}u(x),v) \nonumber \\
& \leq  |\partial \E|^2(\SS_\delta w) + \frac{1}{\delta^2}\sfd^2(u(x),w), 
\end{align} 
where we used both the contractivity \eqref{eq:contraction} of $\SS$ and the monotonicity of the slope \eqref{eq:monotonicity-slope}.
Defining $g_\E (x)=|\partial \E|(\SS_{\delta + \phi(x)} u(x))$, we see that $g_\E \in L^2(\Omega)$.
We now control the variations of $\tilde{E}^\delta$:
starting from \eqref{eq:doubled-slope} and then using the triangle inequality, one has
\begin{multline*}
|\tilde{E}^\delta(x) - \tilde{E}^\delta(y)|  \leq (g_\E(x) + g_\E(y)) \sfd(\SS_{\delta + \phi(x)} u(x), \SS_{\delta + \phi(y)} u(y)) \\
\leq (g_\E(x) + g_\E(y)) \Big( \sfd(\SS_{\delta + \phi(x)} u(x), \SS_{\delta + \phi(x)} u(y)) + \sfd(\SS_{\delta + \phi(x)} u(y), \SS_{\delta + \phi(y)} u(y)) \Big).
\end{multline*}  
We then use \eqref{eq:contraction} followed by \eqref{eq:hajlasz} to handle the first term, and \eqref{eq:asymptotics} to handle the second one.
We get the existence of a negligible set $N \subset \Omega$ such that for $x,y \in \Omega \setminus N$,
\begin{align*}
|\tilde{E}^\delta(x) - \tilde{E}^\delta(y)| & \leq (g_\E(x) + g_\E(y)) \Big( \sfd(u(x),  u(y)) + |\phi(x) - \phi(y)| (g_\E(x) + g_\E(y)) \Big) \\
&  \leq |x-y| (g_\E(x) + g_\E(y)) \Big( g_u(x) + g_u(y) + \| \nabla \phi \|_\infty (g_\E(x) + g_\E(y)) \Big) \\
& \leq |x-y| (g(x) + g(y))
\end{align*}  
provided we define $g(x) = (1 +  \| \nabla \phi \|_\infty )g_\E(x)^2  + g_u(x)^2$. By the considerations above the function $g$ belongs to $L^1(\Omega)$, and \cite[Lemma 5.1.7]{ambrosio2004topics} enables to conclude to $\nabla \tilde{E}^\delta \in L^1(\Omega)$ (with actually $|\nabla \tilde{E}^\delta(x)|\leq 2g(x)$).

%\begin{align*}
%|\tilde{E}^\delta(x) - \tilde{E}^\delta(y)| & \leq (g_\E(x) + g_\E(y)) \sfd(\SS_{\delta + \phi(x)} u(x), \SS_{\delta + \phi(y)} u(y)) \\
%& \leq  (g_\E(x) + g_\E(y)) \sqrt{\sfd^2(u(x), u(y)) + 2 (\phi(x) - \phi(y)) (\tilde{E}^\delta(x) - \tilde{E}^\delta(y))} \\
%& \leq |x-y| (g_\E(x) + g_\E(y)) \sqrt{(g_u(x) + g_u(y))^2 + 2 \| \nabla \phi \|_\infty  (|\tilde{E}^\delta(x)| + |\tilde{E}^\delta(y)|)}.
%\end{align*}  
%Using the subadditivity of the square root, we can reorganize the terms as
%\[
%\begin{split}
%|\tilde{E}^\delta(x) - \tilde{E}^\delta(y)| & \leq |x-y| (g_\E(x) + g_\E(y)) \left[ g_u(x) + g_u(y) + \sqrt{2} \left( \sqrt{|\tilde{E}^\delta|}(x) + \sqrt{|\tilde{E}^\delta|}(y)| \right) \right] \\
%& \leq |x-y| (g(x) + g(y)),
%\end{split}
%\]
%provided we define $g = (1+\sqrt{2})g_\E^2 + g_u^2 + \sqrt{2} |\tilde{E}^\delta|$. By the considerations above the function $g$ belongs to $L^1(\Omega)$, and \cite[Lemma 5.1.7]{ambrosio2004topics} enables to conclude to $\nabla \tilde{E}^\delta \in L^1(\Omega)$, actually it is bounded by $2g$.

Lastly, we need to justify the boundary conditions. Let us take $\Z$ a smooth vector field transverse to $\partial \Omega$. As in Section~\ref{s2} we denote by $x_t$ the solution to $\dot{x}_t = \Z(x_t)$ starting from $x_0 \in \partial \Omega$. Fixing $x_0 \in \partial \Omega$ such that $(u(x_t))_{t \geq 0}$ is continuous and converges to $u^b(x_0)$ as $t \to 0$, it suffices to show that $\tilde{E}^\delta(x_t)$ converges to $\E(\SS_\delta u^b(x_0))$ as $t \to 0$. As $(u(x_t))_{t \geq 0}$ is continuous, it is bounded at least for $t \leq 1$ thus \eqref{eq:control_dE_delta} yields that $g_\E$ is bounded on the image of $(u(x_t))_{t \in [0,1]}$. Let us denote by $C$ the upper bound. We simply use \eqref{eq:doubled-slope} and then the triangle inequality, and finally the contraction property \eqref{eq:contraction} to estimate
\begin{align*}
| \tilde{E}^\delta(x_t) & - \E(\SS_\delta u^b(x_0)) | \leq 2 C  \sfd(\SS_{\delta + \phi(x_t)} u(x_t), \SS_{\delta} u^b(x_0) ) 
\\
& \leq 2C \Big( \sfd(\SS_{\delta + \phi(x_t)} u(x_t), \SS_{\delta + \phi(x_t)} u^b(x_0) ) + \sfd(\SS_{\delta + \phi(x_t)} u^b(x_0), \SS_{\delta} u^b(x_0) ) \Big) 
\\
& \leq 2C \Big( \sfd(u(x_t), u(x_0)) + \sfd(\SS_{\delta + \phi(x_t)} u^b(x_0), \SS_{\delta} u^b(x_0) )  \Big).
\end{align*}  
As $t \to 0$, the first distance goes to $0$ by assumption, and the second one does too owing to $\delta + \phi(x_t) \to \delta + \phi(x_0) = \delta$ together with the continuity of $s \mapsto \SS_s u^b(x_0)$.
\end{proof}

\begin{prop}
\label{prop:control_energy_entropy}
Let $\phi \in C^2(\overline{\Omega})$ be non-negative with $\phi = 0$ on $\partial\Omega$, and take any $u \in H^1(\Omega,\X)$ with boundary value $u^b$ such that $\E(u^b)\in L^1(\partial\Omega)$.
For fixed $\delta>0$, set
\[
\tilde u^\delta(x) := \SS_{\delta+\phi(x)} u(x).
\]
Then $\tilde u^\delta \in H^1(\Omega,\X)$ and
\begin{equation}
\label{eq:control_energy_entropy_delta}
\Dir(\tilde u^\delta )
+\int_\Omega(-\Delta \phi)(x)\E(\tilde u^\delta(x))\,\rd x 
\leq 
\Dir(u) + \int_{\partial\Omega}\left(-\frac{\partial\phi}{\partial\sfn}\right)\E(\SS_\delta u^b)\,\rd\sigma,
\end{equation}
where $\sigma=\mathcal{H}^{d-1}\mrest \partial\Omega$.
\end{prop}

\begin{proof}
Note that $\tilde E^\delta=\E\circ \tilde u^\delta\in L^1(\Omega)$ by Lemma~\ref{lem:E_L1}, so the integral in the left-hand side is well defined.
Also, as $u^b$ belongs to $L^2(\partial \Omega,\X)$ as the trace of an $H^1(\Omega,\X)$ function, see \eqref{eq:trace-op}.
Thus Lemma~\ref{lem:E_L1} (with $\phi \equiv 0$, $\Xi=\partial\Omega$ and $\mu=\mathcal{H}^{d-1}\mrest \partial\Omega$) yields $\E(\SS_\delta u^b) \in L^1(\partial \Omega)$ at least for $\delta > 0$.

%Reasonning exactly as in the beginning of the proof of Lemma~\ref{lem:E_W11}, we see that $\E(S_\delta u^b) \in L^1(\partial \Omega)$ at least for $\delta > 0$.
%by \eqref{eq:monotonicity-entropy} we know that $0\leq \E(\SS_\delta u^b(x))\leq \E(u^b(x))$, and since we assume that this latter quantity is integrable we see that $\E(\SS_\delta u^b)\in L^1(\partial\Omega)$ in the right-hand side of \eqref{eq:control_energy_entropy_delta}.
As a consequence, if \eqref{eq:control_energy_entropy_delta} holds true, then in particular $\tilde u^\delta \in H^1(\Omega,\X)$ by the finiteness of the right-hand side. We are thus left to prove the validity of the estimate \eqref{eq:control_energy_entropy_delta}.

By Lemma \ref{lem:fundamental} we have
\[
\frac 12 \sfd^2(\tilde u^\delta (x),\tilde u^\delta (y))
+ (\phi(x)-\phi(y))\big(\E(\tilde u^\delta (x))-\E(\tilde u^\delta (y))\big)
\leq 
\frac 12 \sfd^2(u(x),u(y))
\]
for all $x,y\in\Omega$, paying attention to the fact that both $\E(\tilde u^\delta (x))$ and $\E(\tilde u^\delta (y))$ are finite because here $\phi+\delta\geq \delta>0$.
Dividing by $C_d \eps^{d+2}$, being $C_d$ the constant defined in \eqref{eq:dirichlet_eps}, and integrating in $x,y$, we get exactly
\[
\Dir_\eps(\tilde u^\delta) + \frac{1}{C_d \eps^d}\iint_{\Omega\times\Omega} \frac{\phi(y)-\phi(x)}{\eps}\,\frac{\tilde E^\delta(y)-\tilde E^\delta(x)}{\eps}\mathds{1}_{|x-y|\leq \eps}\, \rd x\rd y
\leq
\Dir_\eps(u).
\]
For fixed $\delta$ we can now pass to the limit $\eps\to 0$.
By \eqref{eq:dirichlet} the two $\eps$-approximate Dirichlet energies converge to the respective Dirichlet energies.

On the other hand, the limit of the second term of the left-hand side can be guessed quite easily from a Taylor expansion. The rigorous justification is a matter of real analysis that we postpone to Lemma~\ref{lem:IPP_W11} in the appendix:
We are here in position to apply this Lemma, because for fixed $\delta>0$ the function $\tilde E^\delta=\E\circ\tilde u^\delta$ belongs to $W^{1,1}(\Omega)$ by Lemma~\ref{lem:E_W11}. By definition the constant $C_d$ in \eqref{eq:dirichlet_eps} is the same as in Lemma~\ref{lem:IPP_W11}, hence we conclude that
\[
\Dir(\tilde u^\delta)
+\int_\Omega\nabla\phi\cdot \nabla\tilde E^\delta\,\d x
\leq 
\Dir(u).
\]
Finally, since $\phi\in C^2(\overline{\Omega})$ and $\tilde E^\delta\in W^{1,1}(\Omega) \subset BV(\Omega)$ we can integrate by parts in the BV sense \cite[Theorem 5.6]{EG}
\[
\int_\Omega\nabla\phi\cdot \nabla\tilde E^\delta\,\d x
=
-\int_\Omega \Delta \phi \,\tilde E^\delta \,\rd x + \int_{\partial\Omega}\frac{\partial\phi}{\partial\sfn} \left(\tilde E^\delta\right)^b\,\rd\sigma.
\]
Lemma~\ref{lem:E_W11} guarantees that the boundary trace is exactly $(\tilde E^\delta)^b=\E(\SS_\delta u^b)$, hence \eqref{eq:control_energy_entropy_delta} follows and the proof is complete.
%Owing to the assumption that $-\frac{\partial\phi}{\partial\sfn}>0$ a.e.\ on $\partial\Omega$, Lemma~\ref{lem:E_W11} guarantees that the boundary trace is exactly $(\tilde E^\delta)^b=\E(\SS_\delta u^b)$ and the proof is complete.
\end{proof}

The next step is to remove the regularization parameter $\delta>0$.
At this stage this temporarily imposes either an additional super-harmonicity condition on the test function $\phi$ (allowing to apply a monotone convergence, see below), or an extra $L^1$-regularity assumption on $\E\circ \tilde u$. 
We will actually establish this regularity in full generality later on, so one can essentially think of this statement as holding for any $\phi\geq 0$ smooth enough. 

\begin{prop}
\label{prop:phi_superharmonic}
Let $u \in H^1(\Omega,\X)$ with boundary value $u^b$ be such that $\E(u^b)\in L^1(\partial\Omega)$. Fix $\phi \in C^2(\overline{\Omega})$ a non-negative super-harmonic function ($-\Delta\phi\geq 0$) vanishing on the boundary, and set
\[
\tilde u(x) := \SS_{\phi(x)} u(x).
\]
Then $\tilde u \in H^1(\Omega,\X)$ with trace $(\tilde u)^b=u^b$, and satisfies
\begin{equation}
\label{eq:control_energy_entropy}
\Dir(\tilde u )+\int_\Omega(-\Delta \phi)(x)\E(\tilde u(x))\,\d x 
\leq 
\Dir(u) + \int_{\partial\Omega}\left(-\frac{\partial\phi}{\partial\sfn}\right)\E(u^b)\,\d\sigma.
\end{equation}
If one assumes in addition that $\E(\tilde u) \in L^1(\Omega)$, then the hypothesis $-\Delta \phi \geq 0$ is no longer needed. 
\end{prop}

Note that the right-hand side is finite, since by assumption $\E(u^b)\in L^1(\partial\Omega)$. On the other hand, the negative part of the integrand in the left-hand side is in $L^1(\Omega)$ if either $\phi$ is super-harmonic or $\E(\tilde{u}) \in L^1(\Omega)$. Thus, it is actually part of the statement that both the integral and the Dirichlet energy in the left-hand side are finite. Also, as tempting as it might be, we do not claim any $W^{1,1}$-regularity of $\tilde E := \E \circ \tilde u$ or that its trace is $\E(u^b)$.

\begin{proof}
Take $0 < \delta \leq 1$ and let as before
\[
\tilde u^\delta(x):=\SS_{\delta+\phi(x)} u(x)=\SS_{\delta}\tilde u(x).
\]
The strategy is to pass to the liminf as $\delta \downarrow 0$ in \eqref{eq:control_energy_entropy_delta}. To this end, note first that $\tilde u^\delta\to\tilde u$ at least pointwise a.e.\ and that by \eqref{eq:monotonicity-entropy} and lower semicontinuity of $\E$ we have monotone convergence $\E(\tilde u^\delta)\nearrow \E(\tilde u)$ at least pointwise a.e.\ as $\delta\searrow 0$.
%Since $\phi$ grows linearly at the boundary we are exactly in position of applying Proposition~\ref{prop:control_energy_entropy}, and we show now that all the terms pass to the $\liminf$ as $\delta\to 0$ in \eqref{eq:control_energy_entropy_delta}.

By \eqref{eq:uniform-bound}, which holds true also for $\delta=0$, we see that for any arbitrary $w \in \X$
\[
\sfd(\tilde{u}^\delta(x),\tilde{u}(x)) \leq \sfd(\tilde{u}^\delta(x),w) + \sfd(\tilde{u}(x),w) \leq 2\sfd(u(x),w) + C_{\phi, w}
\]
for all $\delta < 1$.
This uniform bound, the fact that $\tilde u^\delta \to \tilde u$ a.e., and Lebesgue's dominated convergence theorem as well as the very definition \eqref{eq:L2-distance} of $\sfd_{L^2}$ then yield that, in fact, $\tilde u^\delta \to \tilde u$ in $L^2(\Omega,\X)$.
Whence by lower semicontinuity of the Dirichlet energy \eqref{eq:lsc-dirichlet} we see that
\[
\Dir(\tilde u) \leq \liminf_{\delta \downarrow 0} \Dir(\tilde u^\delta).
\]
Since $\E(\tilde u^\delta)\nearrow \E(\tilde u)$, $(-\Delta\phi)\geq 0$, and $\E(\tilde u^\delta)$ belongs to $L^1(\Omega)$ for any $\delta > 0$, Beppo Levi's monotone convergence theorem guarantees that the integral in the left-hand side of \eqref{eq:control_energy_entropy_delta} passes to the limit.
If we no longer assume $(-\Delta\phi)\geq 0$ and require instead that $\E(\tilde u) \in L^1(\Omega)$, then we can rely solely on the pointwise convergence of $\E(\tilde u^\delta)$ to $\E(\tilde u)$ as well as the monotonicity $\E(\tilde u)\leq \E(\tilde u^\delta)\leq \E(\tilde u ^1)$ for $0<\delta\leq 1$ and thus
\begin{equation*}
\left| (-\Delta\phi) \E(\tilde u^\delta) \right| \leq \|\Delta \phi \|_\infty \max\big\{ |\E(\tilde u)|, |\E(\tilde u^1)| \big\}. 
\end{equation*} 
As the function $\E(\tilde u)$ is assumed to be in $L^1(\Omega)$ while $\E(\tilde u^1) \in L^1(\Omega)$ by Lemma~\ref{lem:E_L1}, Lebesgue's dominated convergence theorem allows passing to the limit as $\delta \downarrow 0$ in this case too.

In the right-hand side, observe that $\phi\geq 0$ in $\Omega$ implies $-\frac{\partial\phi}{\partial \sfn}\geq 0$ on the boundary.
The monotonicity $\E(\SS_\delta u^b)\nearrow \E(u^b)$ then allows to take the limit similarly.
Here we use that $\E(\SS_\delta u^b) \in L^1(\partial\Omega)$ for any $\delta > 0$, which is a consequence of Lemma~\ref{lem:E_L1}.
Thus taking the $\liminf$ as $\delta\to 0$ in \eqref{eq:control_energy_entropy_delta} results exactly in \eqref{eq:control_energy_entropy}, which as already discussed grants in particular that $\tilde u \in H^1(\Omega,\X)$.
 
Finally, let us check that $\tilde u$ has trace $u^b$.
To this aim, fix any transversal vector field $\Z$ pointing inward on $\partial\Omega$, and denote again $x_t$ the integral curve starting from $x_0 \in \partial\Omega$.
Since $u\in H^1(\Omega,\X)$ by assumption and $\tilde u \in H^1(\Omega,\X)$ by the previous argument, recalling the construction of the trace operator \eqref{eq:trace-op} discussed in Section \ref{s2} we know that, for a.e.\ $x_0 \in \partial\Omega$, the curve $u(x_t)$ is H\"older continuous and converges to $u^b(x_0)$ as $t\to 0$.
Hence, to prove that $\tilde u$ has trace $u^b$, it is enough to check that $\tilde{u}(x_t)$ also converges to $u^b(x_0)$ as $t\to 0$.
By the contractivity property \eqref{eq:contraction} and continuity of $\EVI_0$-gradient flows we can write
\begin{align*}
\sfd(\tilde u(x_t),u^b(x_0))
& = \sfd(\SS_{\phi(x_t)}u(x_t),u^b(x_0)) \\
& \leq \sfd(\SS_{\phi(x_t)}u(x_t),\SS_{\phi(x_t)}u^b(x_0)) + \sfd(\SS_{\phi(x_t)}u^b(x_0),u^b(x_0)) \\
& \leq \sfd(u(x_t),u^b(x_0)) + \sfd(\SS_{\phi(x_t)}u^b(x_0),u^b(x_0)) \to 0,
\end{align*}
because $\phi(x_t)\to\phi(x_0)=0$ on the boundary. This completes the proof of our claim, hence of the proposition.  
%
%
%
%
%
%
% the curves $u(x_t)$ and $\tilde u(x_t) = \SS_{\phi(x_t)}\gamma_t$ are both H\"older continuous in a right neighbourhood of $t=0$ with
%\[
%\lim_{t \to 0} u(x_t) = u^b(x_0) \qquad\textrm{and}\qquad \tilde \lim_{t \to 0} \tilde u(x_t) = (\tilde u)^b(x_0).
%\]
%By the contractivity property \eqref{eq:contraction} and continuity of $\EVI_0$-gradient flows we have moreover
%\[
%\begin{split}
%\sfd(\tilde u(x_t),u^b(x_0))
%& = \sfd(\SS_{\phi(x_t)}u(x_t),u^b(x_0)) \leq \sfd(\SS_{\phi(x_t)}u(x_t),\SS_{\phi(x_t)}u^b(x_0)) + \sfd(\SS_{\phi(x_t)}u^b(x_0),u^b(x_0))
%\\
%& \leq \sfd(u(x_t),u^b(x_0)) + \sfd(\SS_{\phi(x_t)}u^b(x_0),u^b(x_0)) \to 0,
%\end{split}
%\]
%because $\phi(x_t)\to\phi(x_0)=0$ on the boundary. Hence $\tilde u(x_t) \to u^b(x_0)$ and this fact together with $\tilde \gamma_0 = \lim_{t \to 0} \tilde \gamma_t = (\tilde u)^b(x_0)$ yields $u^b(x_0) = (\tilde u)^b(x_0)$ and the proof is complete.
\end{proof}

\begin{cor}
\label{cor:E_is_L1}
Let $u\in H^1(\Omega,\X)$ be harmonic with boundary value $u^b$, and $\E(u^b)\in L^1(\partial\Omega)$.
Then $\E(u)\in L^1(\Omega)$.
\end{cor}

\begin{proof}
Let $\psi$ be the unique $C^2(\overline{\Omega})$-solution of
\begin{equation}
\label{eq:def_psi}
 \begin{cases}
  -\Delta \psi=1 & \mbox{in }\Omega,\\
  \psi=0 & \mbox{on }\partial\Omega.
 \end{cases}
\end{equation}
Note that by the maximum principle we have $\psi\geq 0$ in $\Omega$, hence in particular $-\frac{\partial\psi}{\partial \sfn}\geq 0$ on $\partial\Omega$.
For small $\lambda>0$ let $\phi(x):=\lambda \psi(x)$, write
\[
\tilde u^\lambda(x) := \SS_{\lambda\psi(x)}u(x),
\]
and observe that
\[
\tilde u^\lambda(x)\to u(x)
 \qquad \mbox{a.e. as }\lambda \searrow 0.
\]
According to Proposition~\ref{prop:phi_superharmonic} we know that, for any fixed $\lambda>0$, $\tilde u^\lambda\in H^1(\Omega,\X)$ has the same boundary trace $u^b$ as $u$, and is thus an admissible competitor for the Dirichlet problem. Hence $\Dir(u)\leq \Dir(\tilde u^\lambda)$ in \eqref{eq:control_energy_entropy} and thus
\[
\int_\Omega(-\lambda\Delta \psi)(x)\E(\tilde u^\lambda(x))\,\d x 
\leq 
\int_{\partial\Omega}\left(-\lambda\frac{\partial\psi}{\partial \sfn}\right)\E(u^b)\,\d\sigma
\]
for any $\lambda$. Dividing by $\lambda>0$ and recalling that $-\Delta\psi=1$ we get
\[
\int_\Omega\E(\tilde u^\lambda(x))\,\d x 
\leq 
\int_{\partial\Omega}\left(-\frac{\partial\psi}{\partial \sfn}\right)\E(u^b)\,\d\sigma
\leq C_\Omega \int \E(u^b)^+\,\d\sigma<+\infty,
\]
where $C_\Omega>0$ is a uniform upper bound for $-\frac{\partial\psi}{\partial \sfn}(x)>0$ that only depends on the domain $\Omega$.
Finally, note that $\lambda\psi(x)>0$ is monotone increasing in $\lambda$.
By the monotonicity property \eqref{eq:monotonicity-entropy} we have that, at least for $\lambda \leq 1$, $\E(\tilde u^1(x))\leq \E(\tilde u^\lambda(x))\nearrow \E(u(x))$ as $\lambda \searrow 0$ and the claim finally follows by Beppo Levi's monotone convergence combined with the integrability of $\E(\tilde u^1)^-$, see Lemma~\ref{lem:E_L1}.
\end{proof}

\begin{rmk}
\label{rk:domain_regularity}
We have used the $C^{2,\alpha}$-regularity of $\partial\Omega$ in the proof of Corollary~\ref{cor:E_is_L1} by claiming that $\psi$, defined as a solution of \eqref{eq:def_psi}, is $C^2$ up to the boundary, see \cite[Theorem 6.14]{GilTru15}. This is known to fail if, for instance, the domain has only a Lipschitz boundary: see \cite[Theorem A]{jerison1995inhomogeneous}. As shown in \cite[Theorem 1.2 and Remark 1.3]{Costabel19}, $C^1$-regularity of the boundary is not sufficient, either.
However, as the reader can check, if for some reason one could justify the existence of a smooth non-negative function $\psi$ that vanishes on $\partial \Omega$ and such that $- \Delta \psi \geq c > 0$ uniformly on $\Omega$, then Corollary~\ref{cor:E_is_L1} would hold, as well as Theorem~\ref{thm:main}.
For instance, our approach covers the hypercube in $\R^d$.     
\end{rmk}

As a consequence of this newly established $L^1$-regularity for $\E(u)$ we will be able to use Proposition~\ref{prop:phi_superharmonic} without the superharmonicity assumption on $\phi$. This will allow us to conclude and prove our main result, Theorem~\ref{thm:main}.

\begin{proof}[Proof of Theorem~\ref{thm:main}]
The $L^1$-regularity has already been proved in Corollary~\ref{cor:E_is_L1}. In order to establish \eqref{eq:control_energy_entropy_improved_harmonic}, fix any $\varphi\geq 0$ as in our statement.
We take $\phi=\lambda \varphi$ and write $\tilde u^\lambda(x) := \SS_{\lambda \varphi(x)}u(x)$. We note that $\E(\tilde u^\lambda) \leq \E(u)$ by monotonicity so that $\E(\tilde u^\lambda)^+ \in L^1(\Omega)$, while Lemma~\ref{lem:E_L1} yields $\E(\tilde u^\lambda)^- \in L^1(\Omega)$. We can then apply Proposition~\ref{prop:phi_superharmonic} without the assumption $- \Delta \phi \geq 0$, as $\E(\tilde u^\lambda) \in L^1(\Omega)$, and observe that $\tilde u^\lambda$ is an admissible competitor in the Dirichlet problem (with data $u^b$), so that plugging this information into \eqref{eq:control_energy_entropy} gives
\[
\int_\Omega(-\Delta \varphi)(x)\E(\tilde u^\lambda(x))\,\d x 
\leq 
\int_{\partial\Omega}\left(-\frac{\partial\varphi}{\partial \sfn}\right)\E(u^b)\,\d\sigma.
\]
Exploiting as before the monotone convergence $\E(\tilde u^1(x))\leq \E(\tilde u^\lambda(x))\nearrow \E(u(x))$ with $\E(u)\in L^1(\Omega)$ as well as $\E(\tilde u^1(x))^- \in L^1(\Omega)$, ensured by Lemma~\ref{lem:E_L1},  we see that $\E(\tilde u^\lambda)\to\E(u)$ in $L^1(\Omega)$ as $ \lambda \to 0$ and \eqref{eq:control_energy_entropy_improved_harmonic} follows.
\end{proof}

\begin{rmk}
\label{rk:local_min}
In the proofs of Corollary~\ref{cor:E_is_L1} and Theorem~\ref{thm:main}, we only used that $u$ is harmonic to write $\Dir(\tilde{u}^\lambda) \leq \Dir(u)$ for $\lambda$ small enough, and, as hinted in the proof of Proposition~\ref{prop:phi_superharmonic}, it is easy to see that $\tilde{u}^\lambda$ converges strongly to $u$ in $L^2(\Omega,\X)$ in both cases.
Thus we could relax the assumption ``$u$ harmonic'' into ``$u$ is a local minimizer of $\Dir$ in $H^1_u(\Omega,\X)$ for the strong $L^2(\Omega,\X)$ topology'' and the results would remain valid.
\end{rmk}

\begin{proof}[Proof of Theorems~\ref{th:Linfty} and \ref{rk:Lpmax}]
We argue by duality: let $g \in C^\infty(\overline{\Omega})$ be non-negative and let $\varphi = \varphi_g$ be the unique $C^2(\overline{\Omega})$-solution of
\[
\begin{cases}
  -\Delta\varphi_g=g & \mbox{in }\Omega,\\
  \varphi_g =0 & \mbox{on }\partial \Omega.
\end{cases}
\]
Observe that $\varphi_g \geq 0$ on $\Omega$ by the classical maximum principle.
Since $\varphi_g$ meets the requirements of Theorem~\ref{thm:main},
%{\color{red}[Leo]: Nooooo! the elliptic regularity does not guarantee $\phi\in C^2(\bar\Omega)$, even in smooth domains I believe. I don't think this is a serious gap, we as currently written it doens't work}
we get from \eqref{eq:control_energy_entropy_improved_harmonic} and for any pair $(q,q')$ of conjugate exponents
\begin{align*}
\int_\Omega g\,\E(u)\,\d x
& =
\int_\Omega (-\Delta\varphi_g)\E(u)\,\d x
\leq 
\int_{\partial\Omega}\Big(-\frac{\partial\varphi_g}{\partial \sfn}\Big)\E(u^b)\,\d\sigma
\\
& \leq \int_{\partial\Omega}\Big(-\frac{\partial\varphi_g}{\partial \sfn}\Big)\E(u^b)^+ \,\d\sigma \leq \left\|\frac{\partial\varphi_g}{\partial \sfn}\right\|_{L^{q'}(\partial\Omega)} \|\E( u^b)^+\|_{L^{q}(\partial\Omega)}.
\end{align*}
Now assume that we can find another pair $(p,p')$ of conjugate exponents and a constant $C_{L^{p'} \to L^{q'}}$ such that $\left\|\frac{\partial\varphi_g}{\partial \sfn}\right\|_{L^{q'}(\partial\Omega)}  \leq C_{L^{p'} \to L^{q'}} \| g \|_{L^{p'}(\Omega)}$ for all $g \in C^\infty(\overline{\Omega})$: the constant $C_{L^{p'} \to L^{q'}} $ can be interpreted as the operator norm of the ``Data to Neumann'' map in suitable functional spaces.
Given the $L^2(\Omega)$-integrability of $\E(u)^-$ coming from Lemma~\ref{lem:E_L1}, taking the supremum with respect to all non-negative $g \in C(\overline{\Omega})$ with $\| g \|_{L^{p'}(\Omega)} \leq 1$ yields $\| \E(u)^+ \|_{L^p(\Omega)}$ in the left-hand side, see Lemma~\ref{lem:positive_part} (postponed to the appendix to avoid overburdening real analysis).
Thus we conclude: 
\[ 
\| \E(u)^+ \|_{L^p(\Omega)} \leq  C_{L^{p'} \to L^{q'}} \|\E( u^b)^+\|_{L^{q}(\partial\Omega)}. 
\] 

To prove Theorem~\ref{th:Linfty} in the case $\esssup\limits_{x\in\partial\Omega} \E(u^b(x)) = \|\E( u^b)^+\|_{L^{\infty}(\partial\Omega)}$, i.-e.~ $p=q=\infty$, we need only to justify $C_{L^{1} \to L^{1}} \leq 1$.
To this end, observe first that $-\frac{\partial\varphi_g}{\partial \sfn} \geq 0$ everywhere on the boundary due to $\varphi_g\geq 0$, and integrating the equation defining $g$ thus yields
\begin{equation*}
 \left\|\frac{\partial\varphi_g}{\partial \sfn}\right\|_{L^1(\partial\Omega)}
 =\int_{\partial\Omega} -\frac{\partial\varphi_g}{\partial \sfn} \,\d\sigma
  =\int_{\Omega}-\Delta\varphi_g \,\d x
 =\int_{\Omega} g \,\d x = \| g \|_{L^1(\Omega)}.
\end{equation*}
In the case $\esssup\limits_{x\in\partial\Omega} \E(u^b(x))  < 0$, it is enough to shift $\E$ by a constant large enough to make it non-negative.

Then to prove Theorem~\ref{rk:Lpmax} we need to justify $C_{L^{p'} \to L^{q'}} < + \infty$ for the same exponents as in the statement of the Theorem.
For $q > 1$ this follows from standard elliptic regularity and the fact that the boundary trace operator acts from $W^{1,p'}(\Omega)$ into $L^{q'}(\partial\Omega)$, cf. \cite[Section 15.3]{Leoni09}.
More precisely, the continuous image of the trace operator is actually the Besov space $B^{1-\frac 1 {p'},p'}(\partial\Omega)$, see e.g.\ \cite[Section 15.3]{Leoni09}.
Since the dimension of $\partial \Omega$ is $d-1$, this  Besov space is continuously embedded into $L^{\frac{p'(d-1)}{d-p'}}(\partial\Omega)$, cf.\ \cite[Theorem 14.29]{Leoni09}, and an explicit computation shows that  $\frac{p'(d-1)}{d-p'}=q'$.
Finally, let $q=1$ and fix any $p<d'$.
Then $p'>d$, so $W^{1,p'}(\Omega)$ is continuously embedded into $C(\overline\Omega)$.
Thus the trace operator acts continuously from $W^{1,p'}(\Omega)$ into $C(\partial\Omega)$.
We infer that the map $g \mapsto \frac{\partial \varphi_g}{\partial \sfn}$ acts continuously from $L^{p'}(\Omega)$ into $C(\partial\Omega)\subset L^\infty(\partial \Omega)$. 
\end{proof}

\begin{appendices}

\section{Two technical lemmas}

In the proof of Proposition~\ref{prop:control_energy_entropy} we used the following real-analysis Lemma. It would be very easy to prove in the case of $f,g$ smooth functions, here the technical part is to handle the case where $f$ has minimal regularity.

\begin{lem}
\label{lem:IPP_W11}
Let $\Omega\subset \R^d$ be a bounded Lipschitz domain.
For any fixed $f\in W^{1,1}(\Omega)$ and $ g\in C^{1}(\overline{\Omega})$ there holds
\begin{equation}
\lim\limits_{\eps\to 0}\frac{1}{\eps^d}\iint_{\Omega\times\Omega}\frac{f(y)-f(x)}{\eps}\,\frac{ g(y)- g(x)}{\eps}\mathds{1}_{|x-y|\leq \eps}\,\d x\d y
=
C_d
\int_\Omega \nabla f(x)\cdot \nabla g(x)\rd x
\label{eq:IPP_eps}
\end{equation}
with dimensional constant $C_d=\frac 1d\int_{B_1}|z|^2\,\rd z$.
\end{lem}

\begin{proof}
We write first
\begin{align*}
\frac{1}{\eps^d}\iint_{\Omega\times\Omega}\frac{f(y)-f(x)}{\eps}\,&\frac{ g(y)- g(x)}{\eps}\mathds{1}_{|x-y|\leq \eps}\,\d x\d y \\
& = \int_{\Omega}\frac 1{\eps^d}\int_{\Omega\cap B_\eps(x)}\frac{f(y)-f(x)}{\eps}\,\frac{ g(y)- g(x)}{\eps}\,\d y\d x.
\end{align*}
By assumption $\Omega$ is an \emph{extension domain} \cite[Def.\ 3.20 and Prop.\ 3.21]{AFP00}, hence we can extend $f\in W^{1,1}(\Omega)$ to $\bar f\in W^{1,1}(\R^d)$ and in particular
\begin{align*}
\frac{1}{\eps^{d}}\int_{\Omega\cap B_\eps(x)} & \frac{|f(y)-f(x)-\nabla f(x)\cdot(y-x)|}{\eps}\,\d y
\\
& \quad = \frac{1}{\eps^{d}}\int_{\Omega\cap B_\eps(x)}\frac{|\bar f(y)-\bar f(x)-\nabla \bar f(x)\cdot(y-x)|}{\eps}\,\d y
\\
& \quad \leq \frac{1}{\eps^{d}}\int_{B_\eps(x)}\frac{|\bar f(y)-\bar f(x)-\nabla \bar f(x)\cdot(y-x)|}{\eps}\,\d y.
\end{align*}
By \cite[Theorem 1.3]{Spector16} the right-hand side converges to zero in $L^1(\R^d)$ as $\eps \to 0$, hence also in $L^1(\Omega)$, and we conclude that
\begin{equation}
\frac{1}{\eps^d}\int_{\Omega\cap B_\eps(x)} \frac{|f(y)-f(x)-\nabla f(x)\cdot (y-x)|}{\eps}\,\d y\to 0
\hspace{1cm}\mbox{in }L^1(\Omega)\mbox{ as }\eps \to 0.
\label{eq:taylor_f_L1}
\end{equation}
By the mean-value theorem, for any $y\in B_{\eps}(x)$ there exists a point $z_{x,y}$ lying on the segment $[x,y]$ such that $g(y)-g(x)=\nabla g(z_{xy})\cdot (y-x)$, hence
\begin{align*}
\frac{|g(y)-g(x)-\nabla g(x)\cdot (y-x)|}{\eps} 
& =
\frac{| \nabla g(z_{xy})\cdot(y-x)-\nabla  g(x)\cdot (y-x)|}{\eps}\\
& \leq 
| \nabla g(z_{xy})-\nabla  g(x)|\left|\frac{y-x}{\eps}\right|
\\
& \leq 
\sup\limits_{z\in B_\eps(x)}|\nabla g(z)-\nabla g(x)|
\leq C_d \omega(\eps),
\end{align*}
where $\omega$ is any uniform modulus of continuity of $\nabla g\in C(\overline{\Omega})$.
By definition $\omega(\eps)\to 0$ as $\eps\to 0$, hence
\begin{equation}
\sup\limits_{x\in\Omega}\sup\limits_{y\in B_\eps(x)\cap \Omega}\frac{| g(y)- g(x)-\nabla  g(x)\cdot (y-x)|}{\eps}
\to 0
\qqtext{as}\eps \to 0.
\label{eq:taylor_g_Linfty}
\end{equation}
The estimates \eqref{eq:taylor_f_L1} and \eqref{eq:taylor_g_Linfty} allow now to replace rigorously both difference quotients in \eqref{eq:IPP_eps} by their first-order Taylor expansions, whence
\begin{align}
\frac{1}{\eps^d}\iint_{\Omega\times\Omega} & \frac{f(y)-f(x)}{\eps}\,\frac{ g(y)- g(x)}{\eps}\mathds{1}_{|x-y|\leq \eps}\,\d x\d y \nonumber
\\
& \quad\underset{\eps\to 0}{\sim}
\int_{\Omega}\frac{1}{\eps^d}\int_{B_\eps(x) \cap \Omega}
\Big(\nabla f(x)\cdot\frac{y-x}{\eps}\Big)
\,
\Big(\nabla  g(x)\cdot\frac{y-x}{\eps}\Big)\,
\d y\d x \nonumber
\\
& \quad = \int_\Omega \nabla f(x)^t  \Bigg( \frac{1}{\eps^d}\int_{B_\eps(x)\cap\Omega}\left(\frac{y-x}\eps\right)^t\left(\frac{y-x}\eps\right)\rd y\Bigg)\nabla g(x) \,\rd x.
\label{eq:replace_diff_quotient_limit}
\end{align}
Since $B_\eps(x)\cap\Omega=B_\eps(x)$ for any $x\in\Omega$ and $\eps\leq \dist(x,\partial\Omega)$, we have for any fixed $x$ and $\eps$ small enough (depending on $x$)
$$
h_\eps(x):=\frac{1}{\eps^d}\int_{B_\eps(x) \cap \Omega}\left(\frac{y-x}\eps\right)^t\left(\frac{y-x}\eps\right)\rd y
= \frac{1}{\eps^d}\int_{B_\eps(x)}\left(\frac{y-x}\eps\right)^t\left(\frac{y-x}\eps\right)\rd y
= \int_{B_1}zz^t\,\rd z.
$$
In particular this gives the trivial pointwise convergence $h_\eps(x)\to \int_{B_1}zz^t\,\rd z$ as $\eps \to 0$.
It is moreover immediate to check that $h_\eps$ is uniformly bounded, hence by Lebesgue's dominated convergence in \eqref{eq:replace_diff_quotient_limit} with $\nabla f\in L^1,\nabla g\in L^\infty$ we conclude that
\[
\frac{1}{\eps^d}\iint_{\Omega\times\Omega}\frac{f(y)-f(x)}{\eps}\,\frac{ g(y)- g(x)}{\eps}\mathds{1}_{|x-y|\leq \eps}\,\d x\d y \to \int_{\Omega}\nabla f(x)^t\left(\int_{B_1}zz^t \,\d z\right) \nabla g(x)\,\d x
\]
as $\eps \to 0$. A straightforward symmetry argument finally gives that the matrix 
\[
\int_{B_1} zz^t \,\d z = \left(\int_{B_1}z_i^2\,\d z\right)\operatorname{Id}=\left(\frac{1}{d}\int_{B_1}|z|^2\,\d z\right)\operatorname{Id}, 
\]
hence the limiting integral evaluates to $C_d\int_{\Omega}\nabla f(x)\cdot \nabla g(x)\,\rd x$ with $C_d$ as in the statement and the proof is complete.
\end{proof}

Then, in the proof of Theorem~\ref{th:Linfty} and Theorem~\ref{rk:Lpmax}, we use the following Lemma, which is an easy extension of the classical expression of $L^p$ norm by duality.

\begin{lem}
\label{lem:positive_part}
Let $f : \Omega \to (- \infty, + \infty]$ be measurable and assume that $f^- \in L^2(\Omega)$. Then, for any $p \in [1,+\infty]$ and $p'$ its conjugate exponent there holds
\begin{equation*}
\sup_g \left\{ \int_\Omega f(x) g(x) \, \d x \,:\, g \in C^\infty(\overline{\Omega}), \; g \geq 0 \text{ and } \| g \|_{L^{p'}(\Omega)} \leq 1 \right\} = \| f^+ \|_{L^p(\Omega)},
\end{equation*}
where both sides of the equality could be $+ \infty$.
\end{lem}

\begin{proof}
First, note that $\int_\Omega f(x) g(x) \,\d x$ is always well defined in $( - \infty, + \infty]$ as the negative part of $f$ is in $L^2(\Omega) \subset L^1(\Omega)$. Second, note that by standard duality the result would directly hold if the condition $g \in C^\infty(\overline{\Omega})$ were replaced by $g \in L^\infty(\Omega)$. Indeed, in this case it is always better to take $g$ supported on $\{ f \geq 0 \}$.

Thus we only need to prove that  
\begin{multline}
\label{eq:app_to_prove}
\sup_g \left\{ \int_\Omega f(x) g(x) \, \d x \,:\, g \in C^\infty(\overline{\Omega}), \; g \geq 0 \text{ and } \| g \|_{L^{p'}(\Omega)} \leq 1 \right\} \\
 = \sup_g \left\{ \int_\Omega f(x) g(x) \, \d x \,:\, g \in L^\infty(\Omega), \; g \geq 0 \text{ and } \| g \|_{L^{p'}(\Omega)} \leq 1 \right\},
\end{multline}
and even that the left-hand side is larger than the right-hand one, as the other inequality obviously holds. Thus, let us take $g \in L^\infty(\Omega)$ non-negative such that $\| g \|_{L^{p'}(\Omega)} \leq 1$. By a standard convolution, we can find a sequence $(g_n)_{n \in \N}$ of smooth non-negative functions such that $\| g_n \|_{L^{p'}(\Omega)} \leq 1$ and such that $(g_n)_{n \in \N}$ converges to $g$ in any $L^q(\Omega)$, in particular in $L^2(\Omega)$. Thus
\begin{equation*}
\lim_{n \to \infty} \int_\Omega f^-(x) g_n(x)\, \d x = \int_\Omega f^-(x) g(x)\, \d x, \quad \liminf_{n \to \infty} \int_\Omega f^+(x) g_n(x)\, \d x \geq \int_\Omega f^+(x) g(x) \, \d x,
\end{equation*} 
where we used the $L^2(\Omega)$ convergence on one hand, and Fatou's lemma on the other.
We conclude that $\liminf \int_\Omega f(x) g_n(x) \d x \geq \int_\Omega f(x) g(x) \d x$, and this enough to get \eqref{eq:app_to_prove}. 
\end{proof}

As the reader can see, the assumption $f^- \in L^2(\Omega)$ can in fact be relaxed to $f^- \in L^q(\Omega)$ for some $q > 1$.

\label{ap1}

\subsection*{Acknowledgments}
LM was funded by the Portuguese Science Foundation through a personal grant 2020/\\00162/CEECIND as well as the FCT project PTDC/MAT-STA/28812/2017. LT acknowledges financial support from FSMP Fondation Sciences Math\'ematiques de Paris. DV was partially supported by the FCT projects UID/MAT/00324/2020 and PTDC/MAT-PUR/28686/2017. 

\end{appendices}

%%%%%%%%%%%%%%%%%%%%%%%%%%%%%%%%%%%%%%%%%%%%%%%%%%%%%%%%%%%%%%%%%%%%%%%%%%%%%%%%%%%%%%%%%%%%%%%%%%%%%%%%%%%%%%%%%%%%%%%%%%%%%%%%%%%%%%%%%%%%%%%%%%%%%%%%%%%%%%%%

\bibliographystyle{abbrv}
{\small
\bibliography{biblio}}
\end{document}